\documentclass{amsart}
\usepackage{amsmath, amssymb, amsthm, stmaryrd, tikz, pgfplots}
\pgfplotsset{compat=1.15}
\usepackage[T1]{fontenc}
\usepackage[all]{xy}
\usepackage[colorlinks=true]{hyperref}

\theoremstyle{plain}
\newtheorem{theorem}{Theorem}[section]
\newtheorem{proposition}{Proposition}[section]
\newtheorem{corollary}{Corollary}[section]
\newtheorem{lemma}{Lemma}[section]

\theoremstyle{remark}
\newtheorem{remark}{Remark}[section]

\DeclareMathOperator{\R}{\mathbb{R}}

\DeclareMathOperator{\Z}{\mathbb{Z}}
\DeclareMathOperator{\Trace}{Trace}
\DeclareMathOperator{\diag}{diag}

\DeclareMathOperator{\Sym}{Sym}

\DeclareMathOperator{\mcalL}{\mathcal{L}}
\DeclareMathOperator{\mcalM}{\mathcal{M}}

\DeclareMathOperator{\mcalR}{\mathcal{R}}
\DeclareMathOperator{\mcalS}{\mathcal{S}}

\title[Quadratic ineq.\ between the largest eigenvalues of a graph]{Quadratic inequalities between the largest eigenvalues of a graph}

\author{Roland Paulin}
\address{HUN-REN Alfr\'ed R\'enyi Institute of Mathematics, H-1053 Budapest, Re\'altanoda utca 13-15}
\email{paulinroland@gmail.com}
\thanks{The author was supported by the Dynasnet ERC Synergy project (ERC-2018-SYG 810115) and the Hungarian National Research, Development and Innovation Office (Advanced grant 15337)}

\date{\today}

\begin{document}

\begin{abstract}
We prove a sharp quadratic inequality between the largest two eigenvalues $\lambda_1 \ge \lambda_2$ of a graph with $n$ vertices.
We also prove a quadratic inequality between the second and third largest eigenvalues $\lambda_2 \ge \lambda_3$.
These results in particular imply the bounds $\lambda_1 + \lambda_2 \le \frac{8}{7} n - 2$, $\lambda_3 \le \frac{n}{3} - 1$ and $\lambda_2 + \lambda_3 \le \frac{2}{3} n - 2$.
In fact we determine the closure of the set of possible $(\frac{\lambda_1+1}{n}, \frac{\lambda_2+1}{n}) \in \mathbb{R}^2$ and $(\frac{\lambda_2+1}{n}, \frac{\lambda_3+1}{n}) \in \mathbb{R}^2$.
More generally, we prove quadratic bounds in the case of symmetric matrices in $[0,1]^{n \times n}$, and we also give a quadratic bound for two eigenvalues of a symmetric matrix in $[-1,1]^{n \times n}$.
These bounds are proved by transforming the problem into extremal geometric questions in $\mathbb{R}^3$ and $\mathbb{R}^2$.
We use the method of Lagrange multipliers to reduce to special cases with at most five points, and we deal with these special cases directly.
\end{abstract}

\maketitle

\section{Introduction}

In this article we study the first, second and third largest eigenvalues of the adjacency matrix $A_G$ of a graph $G$, and more generally, of a symmetric matrix $A \in [0,1]^{n \times n}$.
Let $\lambda_1 \ge \lambda_2 \ge \dotsc \ge \lambda_n$ be the eigenvalues of the adjacency matrix $A_G$, where $G$ is a finite, simple graph with $n \ge 2$ vertices, and let $s_i = \frac{\lambda_i+1}{n}$.
It is known that $s_1 \le 1$ and $0 \le s_2 \le \frac{1}{2}$ (see \cite{Hong1988}), and recently $s_3 \le \frac{1}{3}$ has been proved by Tang \cite{tang2026sharpupperboundadjacency} (also see \cite{sivashankar2026grapheigenvaluesprojectionconstants}), solving a problem of Nikiforov \cite{Nikiforov2015}.
Recently Kumar, Liu, Monterde, Pragada and Tait \cite{kumar2026maximumspectralsumgraphs} proved $\lambda_1 + \lambda_2 \le \frac{8}{7} n$, solving a conjecture of Ebrahimi~B, Mohar, Nikiforov and Ahmady \cite{EMNA}, and Huang and Wei \cite{huang2026exactmaximumspectralsum} improved this to $s_1+s_2 \le \frac{8}{7}$.

We study the possible values of $(s_1, s_2) \in \R^2$ and $(s_2, s_3) \in \R^2$.
We prove that $s_1^2 + s_1 s_2 + 2 s_2^2 \le s_1 + s_2$ and $2 (s_2^2 + s_2 s_3 + s_3^2) \le s_2+s_3$, and we prove that these bounds are sharp in a sense (Corollary \ref{corA}).
As a consequence we get $s_1+s_2 \le \frac{8}{7}$, $s_2 + s_3 \le \frac{2}{3}$ and $s_3 \le \frac{1}{3}$.
We prove quadratic inequalities not just for graphs, but more generally, for symmetric matrices $A \in [0,1]^{n \times n}$ (Theorem \ref{thmA}).
We also prove a quadratic bound for two eigenvalues of a symmetric matrix $B \in [-1,1]^{n \times n}$ (Theorem \ref{thmB}).
We prove these results as a consequence of extremal geometric statements in $\R^3$ and in $\R^2$ (Propositions \ref{propA12}, \ref{propA23}, \ref{propB} and \ref{propB2}).
In the last extremal geometric problem (in the proof of Proposition \ref{propB2}), we use the method of Lagrange multipliers, and reduce the problem to the case of at most $5$ points (with further restrictions as well), and we deal with these special cases directly.

We use the notation $[n] = \{1, \dotsc, n\}$ for $n \in \Z_{\ge 0}$. 
For symmetric square matrices $A \succeq B$ means $A-B$ is positive semidefinite, and $A \succ B$ means $A-B$ is positive definite.
If $G$ is a finite simple graph with vertex set $[n]$, then $A_G$ denotes its adjacency matrix.
The characteristic polynomial of a matrix $M \in \R^{n \times n}$ is $\chi_M(t) = \det(t I_n - M)$.
Let $\mcalR_{1,2}$ denote the set of $(s_1, s_2) \in \R^2$ such that there is a sequence $(G_i)_{i \ge 1}$ of finite, simple graphs, where $G_i$ has $n_i$ vertices and eigenvalues $\lambda_{i,1} \ge \dotsc \ge \lambda_{i,n_i}$, and $n_i \to \infty$ and $\frac{\lambda_{i,a}+1}{n_i} \to s_a$ as $i \to \infty$ for $a \in \{1,2\}$.
Similarly, let $\mcalR_{2,3}$ denote the set of $(s_2, s_3) \in \R^2$ such that there is a sequence $(G_i)_{i \ge 1}$ of finite, simple graphs, where $G_i$ has $n_i$ vertices and eigenvalues $\lambda_{i,1} \ge \dotsc \ge \lambda_{i,n_i}$, and $n_i \to \infty$ and $\frac{\lambda_{i,a}+1}{n_i} \to s_a$ as $i \to \infty$ for $a \in \{2,3\}$.
Let $\mcalS_{1,2}$ and $\mcalS_{2,3}$ be slight modifications of $\mcalR_{1,2}$ and $\mcalR_{2,3}$, where we do not require $n_i \to \infty$ in the definition.
So $\mcalS_{1,2}$ is the closure of the set $\{(\frac{\lambda_1(G)+1}{|V(G)|}, \frac{\lambda_2(G)+1}{|V(G)|}); \, |V(G)| \ge 2\} \subseteq \R^2$ and $\mcalS_{2,3}$ is the closure of the set $\{(\frac{\lambda_2(G)+1}{|V(G)|}, \frac{\lambda_3(G)+1}{|V(G)|}); \, |V(G)| \ge 3\} \subseteq \R^2$.
Clearly $\mcalR_{1,2} \subseteq \mcalS_{1,2}$ and $\mcalR_{2,3} \subseteq \mcalS_{2,3}$.
In Corollary \ref{corA} we give exact descriptions for $\mcalR_{1,2}$, $\mcalR_{2,3}$, $\mcalS_{1,2}$, $\mcalS_{2,3}$.

A diagram of the structure of the proofs in this article:
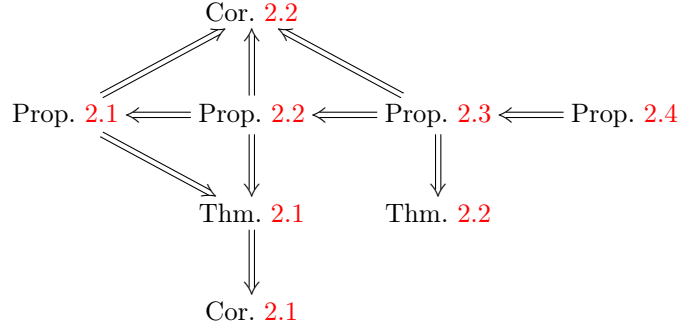
\[\xymatrix{
& \textrm{Cor.\ \ref{corM}} & & \\
\textrm{Prop.\ \ref{propA12}} \ar@{=>}[dr] \ar@{=>}[ru] & \textrm{Prop.\ \ref{propA23}} \ar@{=>}[l]  \ar@{=>}[d] \ar@{=>}[u] & \textrm{Prop.\ \ref{propB}} \ar@{=>}[l] \ar@{=>}[d] \ar@{=>}[lu] & \textrm{Prop.\ \ref{propB2}} \ar@{=>}[l] \\
& \textrm{Thm.\ \ref{thmA}} \ar@{=>}[d] & \textrm{Thm.\ \ref{thmB}} & \\
& \textrm{Cor.\ \ref{corA}} & &
}\]

\section{Statement of the results}

The following theorem gives a quadratic inequality between the first and second largest eigenvalues, and an inequality between the second and third largest eigenvalues of a symmetric matrix in $[0,1]^{n \times n}$.
\begin{theorem} \label{thmA}
Let $n \ge 2$, and let $A \in [0,1]^{n \times n}$ be a symmetric matrix with eigenvalues $\lambda_1 \ge \lambda_2 \ge \dotsc \ge \lambda_n$.
Let $s_i = \frac{\lambda_i}{n}$.
Then $s_1 \in [0,1]$ and $s_2 \le \frac{1}{2}$.
If $s_2 \ge 0$, then
\begin{equation}
s_1^2 + s_1 s_2 + 2 s_2^2 \le s_1 + s_2.
\end{equation}
If $n \ge 3$ and $s_3 \ge 0$, then
\begin{equation}
2 (s_2^2 + s_2 s_3 + s_3^2) \le s_2+s_3.
\end{equation}
\end{theorem}
\begin{remark}
The construction in the proof of Corollary \ref{corA} shows that the bounds here are sharp in some sense (assuming $s_2 \ge 0$ in the first inequality, and $s_3 \ge 0$ in the second one).
The inequality $s_1^2 + s_1 s_2 + 2 s_2^2 \le s_1 + s_2$ can fail if $s_2 < 0$, e.g., for $A = (\begin{smallmatrix} 0 & 1 \\ 1 & 0 \end{smallmatrix})$ we get $(s_1, s_2) = (\frac{1}{2}, -\frac{1}{2})$.
The inequality $2 (s_2^2 + s_2 s_3 + s_3^2) \le s_2+s_3$ can fail if $s_3 < 0$, e.g., for $A = \left(\begin{smallmatrix} 0 & 1 & 0 \\ 1 & 0 & 1 \\ 0 & 1 & 0 \end{smallmatrix}\right)$ we get $(s_1, s_2, s_3) = (\frac{\sqrt{2}}{3}, 0, -\frac{\sqrt{2}}{3})$.
\end{remark}

Figures \ref{fig:s1s2} and \ref{fig:s2s3} show the regions of $(s_1, s_2)$ and $(s_2, s_3)$ defined by the above quadratic inequalities, with the requirements $s_2 \ge 0$ and $s_3 \ge 0$.
As a consequence we get that $s_1 + s_2 \le \frac{8}{7}$, and $s_2+s_3 \le \frac{2}{3}$ and $s_3 \le \frac{1}{3}$ for $n \ge 3$.
The condition $s_2 \ge 0$ is automatically satisfied if $A_{i,i} = 1$ for every $i \in [n]$ (e.g., if $A = I_n + A_G$ for a finite simple graph $G$), because then $1 = \frac{1}{n} \Trace(A) = \sum_{i=1}^n s_i \le s_1 + (n-1) s_2$ and $s_1 \le 1$.

\begin{corollary} \label{corA}
Let $G$ be a finite simple graph with $n \ge 2$ vertices, and let $\lambda_1 \ge \lambda_2 \ge \dotsc \ge \lambda_n$ be the eigenvalues of $A_G$.
Let $s_i = \frac{\lambda_i+1}{n}$.
Then $s_1 \in [0,1]$, $s_2 \in [0,\frac{1}{2}]$,
\[
s_1^2 + s_1 s_2 + 2 s_2^2 \le s_1+s_2,
\]
and if $n \ge 3$, then
\[
2 (s_2^2 + s_2 s_3 + s_3^2) \le s_2+s_3.
\]
If $n \ge 3$ and $G$ is not the path graph with $3$ vertices, then $s_3 \ge 0$.
The quadratic inequalities are sharp in the following sense:
\begin{align*}
\mcalS_{1,2} = \mcalR_{1,2} &= \{(s_1, s_2) \in \R^2; \, 0 \le s_2 \le s_1 \textrm{ and } s_1^2 + s_1 s_2 + 2 s_2^2 \le s_1+s_2 \}, \\
\mcalR_{2,3} &= \{(s_2, s_3) \in \R^2; \, 0 \le s_3 \le s_2 \textrm{ and } 2 (s_2^2 + s_2 s_3 + s_3^2) \le s_2+s_3 \}
\end{align*}
and
\[
\mcalS_{2,3} = \mcalR_{2,3} \cup \left\{\left(\frac{1}{3}, \frac{1-\sqrt{2}}{3}\right)\right\}.
\]
\end{corollary}

We deduce the quadratic inequality for $s_1, s_2$ in Theorem \ref{thmA} from the following $2$-dimensional geometric statement.
\begin{proposition} \label{propA12}
Let $p_1, \dotsc, p_n \in \R_{\ge 0}$, $z_1, \dotsc, z_n \in \R^2$ such that $\sum_i p_i = 1$ and $\sum_i p_i z_i z_i^T \preceq I_2$.
Let $A \in [0,1]^{n \times n}$ be symmetric, and let
\[
\Delta = \sum_{i,j} A_{i,j} p_i p_j z_i z_j^T \in \R^{2 \times 2}.
\]
Let $s_1 \ge s_2$ be the eigenvalues of $\Delta$.
Then $s_1 \le 1$ and $s_2 \le \frac{1}{2}$, and if $s_2 \ge 0$, then
\[
s_1^2 + s_1 s_2 + 2 s_2^2 \le s_1 + s_2.
\]
\end{proposition}

Similarly, we deduce the quadratic inequality for $s_2, s_3$ in Theorem \ref{thmA} from the following $3$-dimensional geometric statement.
\begin{proposition} \label{propA23}
Let $p_1, \dotsc, p_n \in \R_{\ge 0}$, $z_1, \dotsc, z_n \in \R^3$ such that $\sum_i p_i = 1$ and $\sum_i p_i z_i z_i^T \preceq I_3$.
Let $A \in [0,1]^{n \times n}$ be symmetric, and let
\[
\Delta = \sum_{i,j} A_{i,j} p_i p_j z_i z_j^T \in \R^{3 \times 3}.
\]
Let $s_1 \ge s_2 \ge s_3$ be the eigenvalues of $\Delta$.
Then $s_2 \le \frac{1}{2}$, and if $s_3 \ge 0$, then
\[
2 (s_2^2 + s_2 s_3 + s_3^2) \le s_2+s_3.
\]
\end{proposition}
We will prove Proposition \ref{propA12} using Proposition \ref{propA23}.
In the proof of Proposition \ref{propA23}, we use an orthogonal projection from $\R^3$ onto a plane orthogonal to $\sum_i p_i z_i$.
Then we need a plane geometric statement (Proposition \ref{propB}), where instead of $A \in [0,1]^{n \times n}$, we take a symmetric $B \in [-1,1]^{n \times n}$.
This plane geometric statement also allows us to prove the following result about two eigenvalues of such a matrix $B$.

\begin{theorem} \label{thmB}
Let $B \in [-1,1]^{n \times n}$ be a symmetric matrix with eigenvalues $\lambda_1$, \ldots, $\lambda_n$.
Let $\alpha = \frac{\lambda_i}{n}$, $\beta = \frac{\lambda_j}{n}$ for some $i \neq j$.
Then $\alpha^2 + \beta^2 \le 1$, and if $\alpha \beta \ge 0$, then
\begin{equation}
\alpha^2 + \alpha \beta + \beta^2 \le |\alpha+\beta|.
\end{equation}
\end{theorem}
\begin{remark} \label{rem:thmB}
We give two constructions to show that the bound here is sharp in some sense.
First, let $n', n'' \ge 1$, $n = n'+n''$, $B \in \{-1,1\}^{n \times n}$, where $B_{i,j} = -1$ if and only if $\max(i,j) \le n'$.
Then
\[
\chi_B(t) = t^{n-2} (t^2 - (n''-n') t - 2 n' n'') = t^{n-2} (t - n \alpha) (t - n \beta)
\]
with
\[
\alpha = \frac{1}{2n} (n''-n' + \sqrt{(n'+n'')^2+4n'n''}), \ \beta = \frac{1}{2n} (n''-n' - \sqrt{(n'+n'')^2+4n'n''}).
\]
Then $\alpha^2+\beta^2 = 1$.
Here $\alpha > 0 > \beta$ only depend on $\frac{n'}{n''}$, and $(\alpha, \beta) \to (1,0)$ as $\frac{n'}{n''} \to 0$, and $(\alpha, \beta) \to (0,-1)$ as $\frac{n'}{n''} \to \infty$, so we can approximate all points of the circle arc between these (see figure \ref{fig:ab}).
Taking $-B$ we get the opposite circle arc too.

For our second construction, let $n_1, n_2, n_3 \ge 1$, $n = n_1+n_2+n_3$, $B \in \{-1,1\}^{n \times n}$ symmetric, where for $i \le j$ let $B_{i,j} = B_{j,i} = -1$ if and only if $i \le n_1$ and $j > n_1+n_2$.
Then
\[
\chi_B(t) = t^{n-3} (t^3 - n t^2 + 4 n_1 n_2 n_3) = t^{n-3} (t - n \alpha) (t - n \beta) (t - n \gamma)
\]
for some $\alpha \ge \beta \ge \gamma$.
Here $(t - \alpha) (t  - \beta) (t - \gamma) = t^3 - t^2 + c$ with $c = 4 \frac{n_1 n_2 n_3}{n^3} \in (0, \frac{4}{27}]$ (with $c = \frac{4}{27}$ only for $n_1 = n_2 = n_3$), and $(\alpha, \beta, \gamma) \to (1,0,0)$ as $c \to 0$, and $(\alpha, \beta, \gamma) = (\frac{2}{3}, \frac{2}{3}, -\frac{1}{3})$ at $c = \frac{4}{27}$.
So $\alpha \ge \beta > 0 > \gamma$ (because the zeros depend continuously on $c$, see e.g., \cite{Hir-ContOfRoots}), $\gamma = 1 - \alpha - \beta$, and $0 = \alpha \beta + \alpha \gamma + \beta \gamma = \alpha \beta + (\alpha + \beta) (1 - \alpha - \beta) = \alpha + \beta - \alpha^2 - \beta^2 - \alpha \beta$.
We could of course take $(\beta, \alpha)$ instead of $(\alpha, \beta)$, so we can approximate the whole ellipse arc between $(1,0)$ and $(0,1)$ (see figure \ref{fig:ab}).
Taking $-B$ we also get the opposite ellipse arc.
\end{remark}

Figure \ref{fig:ab} shows the closure of the set of attainable $(\alpha, \beta)$.
We will prove this theorem using the following plane geometric statement.

\begin{proposition} \label{propB}
Let $p_1, \dotsc, p_n \in \R_{\ge 0}$, $z_1, \dotsc, z_n \in \R^2$ such that $\sum_i p_i = 1$ and $\sum_i p_i z_i z_i^T \preceq I_2$.
Let $B \in [-1,1]^{n \times n}$ be symmetric, and let
\[
\Delta = \sum_{i,j} B_{i,j} p_i p_j z_i z_j^T \in \R^{2 \times 2}.
\]
Let $\alpha$, $\beta$ be the eigenvalues of $\Delta$.
Then $\alpha^2 + \beta^2 \le 1$, and if $\alpha \beta \ge 0$, then
\[
\alpha^2 + \alpha \beta + \beta^2 \le |\alpha+\beta|.
\]
\end{proposition}

The convex region $\alpha^2+\alpha \beta + \beta^2 \le \alpha+\beta$ with $\alpha, \beta \ge 0$ has a description as an intersection of half planes, which will be useful for us.
\begin{lemma} \label{lem:ellipse}
Let $\alpha, \beta \ge 0$.
Then $\alpha^2 + \alpha \beta + \beta^2 \le \alpha + \beta$ if and only if
\begin{equation}
c_1 \alpha + c_2 \beta \le K_{c_1, c_2} := \frac{1}{3} \left(c_1+c_2 + 2 \sqrt{c_1^2 - c_1 c_2 + c_2^2}\right)
\end{equation}
for every $c_1, c_2 \ge 0$.
\end{lemma}

We will prove Proposition \ref{propB} using the following result.
\begin{proposition} \label{propB2}
Let $p_1, \dotsc, p_n \in \R_{\ge 0}$, $z_1, \dotsc, z_n \in \R^2$ such that $\sum_i p_i = 1$ and $\sum_i p_i z_i z_i^T \preceq I_2$.
Let $c_1, c_2 \ge 0$, $C = \diag(c_1, c_2)$.
Then
\begin{equation}
\sum_{i,j} p_i p_j |z_i^T C z_j| \le K_{c_1, c_2}.
\end{equation}
\end{proposition}
Note that this result easily generalizes to Borel probability measures on $\R^2$ (generalizing the case when the measure is supported on finitely many points $z_i$, where $z_i$ has measure $p_i$).

We can use Propositions \ref{propA12}, \ref{propA23}, \ref{propB} to get the following result.
\begin{corollary} \label{corM}
Let $A \in \R^{n \times n}$ be symmetric, and let $p_1, \dotsc, p_n \in \R_{\ge 0}$, $\sum_i p_i = 1$, $M = (A_{i,j} \sqrt{p_i p_j})_{i,j \in [n]}$.
If $A \in [0,1]^{n \times n}$, $s_1 \ge s_2$ and $(t-s_1) (t-s_2) \mid \chi_M(t)$, then $s_1 \le 1$ and $s_2 \le \frac{1}{2}$, and if also $s_2 \ge 0$, then
\[
s_1^2 + s_1 s_2 + 2 s_2^2 \le s_1 + s_2.
\]
If $A \in [0,1]^{n \times n}$, $s_1 \ge s_2 \ge s_3 \ge 0$ and $(t-s_1) (t-s_2) (t-s_3) \mid \chi_M(t)$, then
\[
2 (s_2^2 + s_2 s_3 + s_3^2) \le s_2+s_3.
\]
If $A \in [-1,1]^{n \times n}$, $\alpha, \beta \in \R$ and $(t-\alpha) (t - \beta) \mid \chi_M(t)$, then $\alpha^2 + \beta^2 \le 1$, and if also $\alpha \beta \ge 0$, then
\[
\alpha^2 + \alpha \beta + \beta^2 \le |\alpha+\beta|.
\]
\end{corollary}
\begin{remark}
If $n \ge 2$, $p_1, \dotsc, p_n \ge 0$ with $\sum_i p_i = 1$, $A \in [0,1]^{n \times n}$ is symmetric, $A_{i,i} = 1$ for every $i$, and the eigenvalues of $M = (A_{i,j} \sqrt{p_i p_j})_{i,j \in [n]}$ are $s_1 \ge \dotsc \ge s_n$, then $s_2 \ge 0$, because $1 = \Trace(M) \le s_1 + (n-1) s_2$ and $s_1 \le 1$.
Note that $M = D^T A D$ with $D = \diag (\sqrt{p_i})_{i \in [n]}$.
So if $p_1, \dotsc, p_n > 0$, then by Sylvester's law of inertia, $M$ and $A$ have the same signature: they have the same rank, and the same number of positive/negative eigenvalues.
So e.g., if $\lambda_3(A) \ge 0$, then $\lambda_3(M) \ge 0$ (and by continuity, we do not even need $p_i > 0$ here, $p_i \ge 0$ suffices).
\end{remark}

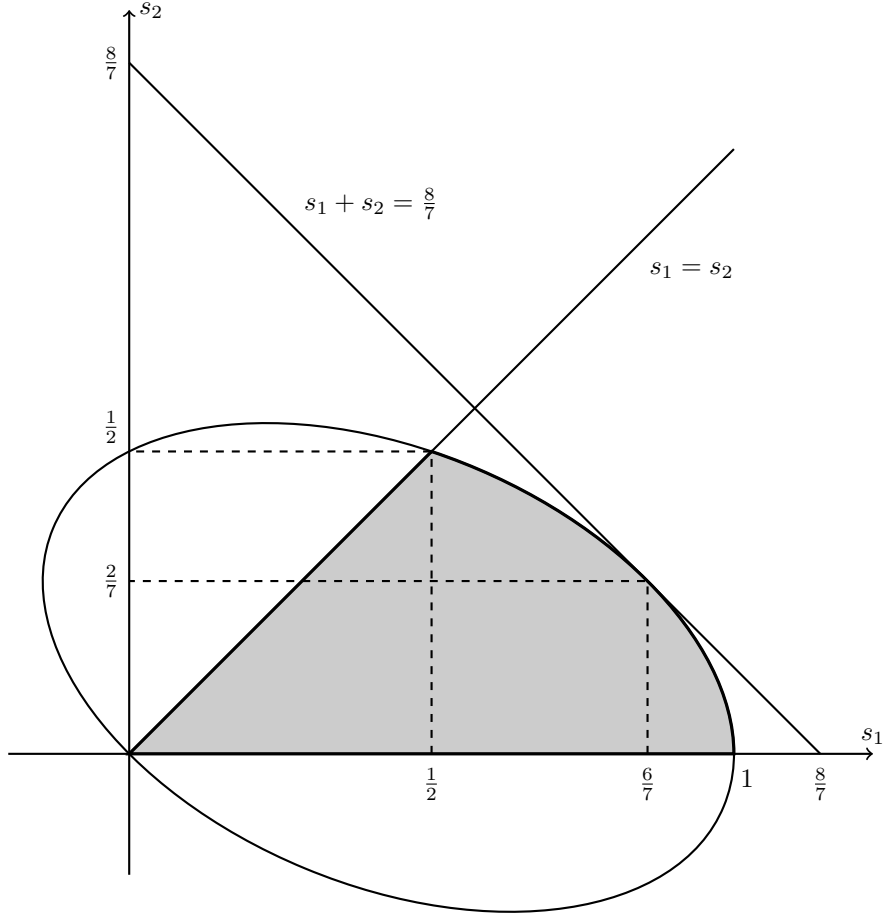
\begin{figure}
\centering
\begin{tikzpicture}[thick, scale=8]
\draw[rotate around={-22.5:(0,0)}, very thick, fill=black!20] ({sqrt(2+sqrt(2))/2},{sqrt(2-sqrt(2))/2}) arc ({asin(sqrt(1/2 - 5/16*sqrt(2)))}: {180 - asin(sqrt(1/2 + 11/32*sqrt(2)))} :{1/sqrt(7/4*(3-sqrt(2)))} and {1/sqrt(7/4*(3+sqrt(2)))}) -- (0,0) -- cycle;
\draw[->] (-0.2,0)--(1.23,0) node[above]{$s_1$};
\draw[->] (0,-0.2)--(0,1.23) node[right]{$s_2$};
\draw[rotate around={-22.5:(0,0)}] ({sqrt(20+2*sqrt(2))/14},{sqrt(20-2*sqrt(2))/14}) ellipse ({1/sqrt(7/4*(3-sqrt(2)))} and {1/sqrt(7/4*(3+sqrt(2)))});
\draw[dashed] (1/2,0) -- (1/2,1/2) -- (0,1/2);
\draw[dashed] (6/7,0) -- (6/7,2/7) -- (0,2/7);
\draw (0,0) -- (1,1);
\draw (8/7,0) -- (0,8/7);
\node at (1/2,-0.05) {$\frac{1}{2}$};
\node at (6/7,-0.05) {$\frac{6}{7}$};
\node at (1.022,-0.04) {$1$};
\node at (8/7,-0.05) {$\frac{8}{7}$};
\node at (-0.03,2/7) {$\frac{2}{7}$};
\node at (-0.03,0.54) {$\frac{1}{2}$};
\node at (-0.03,8/7) {$\frac{8}{7}$};
\node at (0.93,0.8) {$s_1=s_2$};
\node at (0.4,0.91) {$s_1+s_2=\frac{8}{7}$};
\end{tikzpicture}
\caption{The region $\mcalR_{1,2}$}
\label{fig:s1s2}
\end{figure}

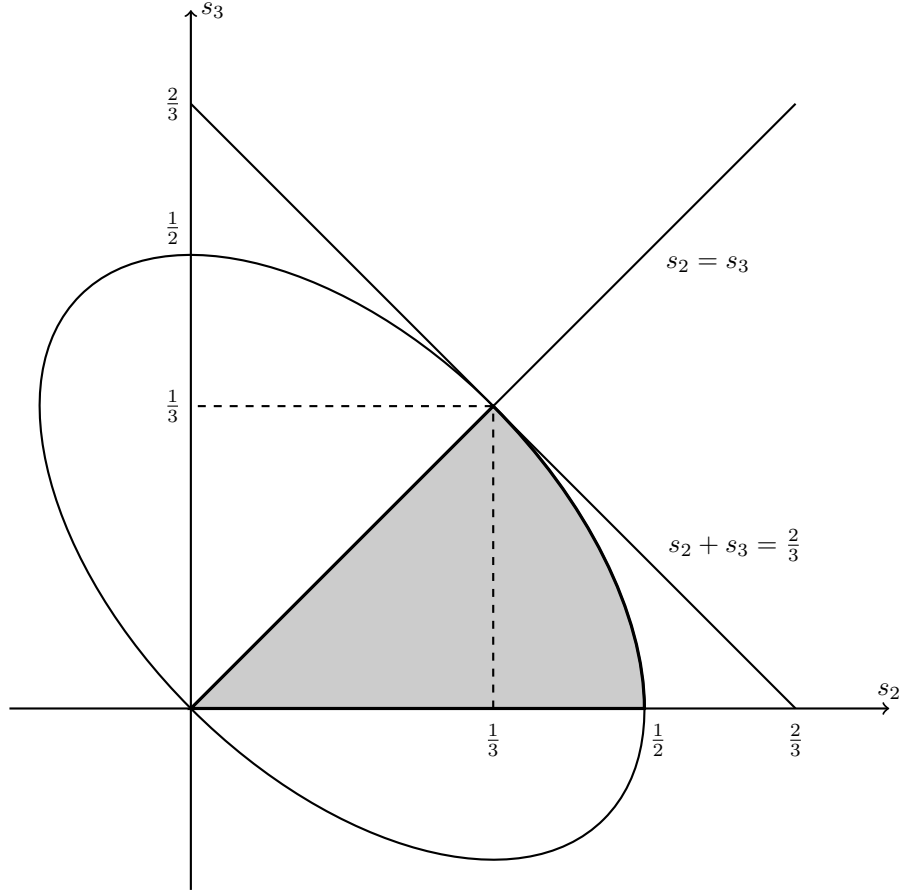
\begin{figure}
\centering
\begin{tikzpicture}[thick, scale=12]
\draw[rotate around={-45:(0,0)}, very thick, fill=black!20] ({1/(2*sqrt(2))},{1/(2*sqrt(2))}) arc (30:90:{sqrt(1/6)} and {sqrt(2)/6}) -- (0,0) -- cycle;
\draw[->] (-0.2,0)--(0.77,0) node[above]{$s_2$};
\draw[->] (0,-0.2)--(0,0.77) node[right]{$s_3$};
\draw[rotate around={-45:(0,0)}] (0,{sqrt(2)/6}) ellipse ({sqrt(1/6)} and {sqrt(2)/6});
\draw[dashed] (1/3,0) -- (1/3,1/3) -- (0,1/3);
\draw (0,0) -- (2/3,2/3);
\draw (2/3,0) -- (0,2/3);
\node at (1/3,-0.035) {$\frac{1}{3}$};
\node at (0.515,-0.035) {$\frac{1}{2}$};
\node at (2/3,-0.035) {$\frac{2}{3}$};
\node at (-0.02,1/3) {$\frac{1}{3}$};
\node at (-0.02,0.53) {$\frac{1}{2}$};
\node at (-0.02,2/3) {$\frac{2}{3}$};
\node at (0.57,0.49) {$s_2=s_3$};
\node at (0.6,0.18) {$s_2+s_3=\frac{2}{3}$};
\end{tikzpicture}
\caption{The region $\mcalR_{2,3}$}
\label{fig:s2s3}
\end{figure}

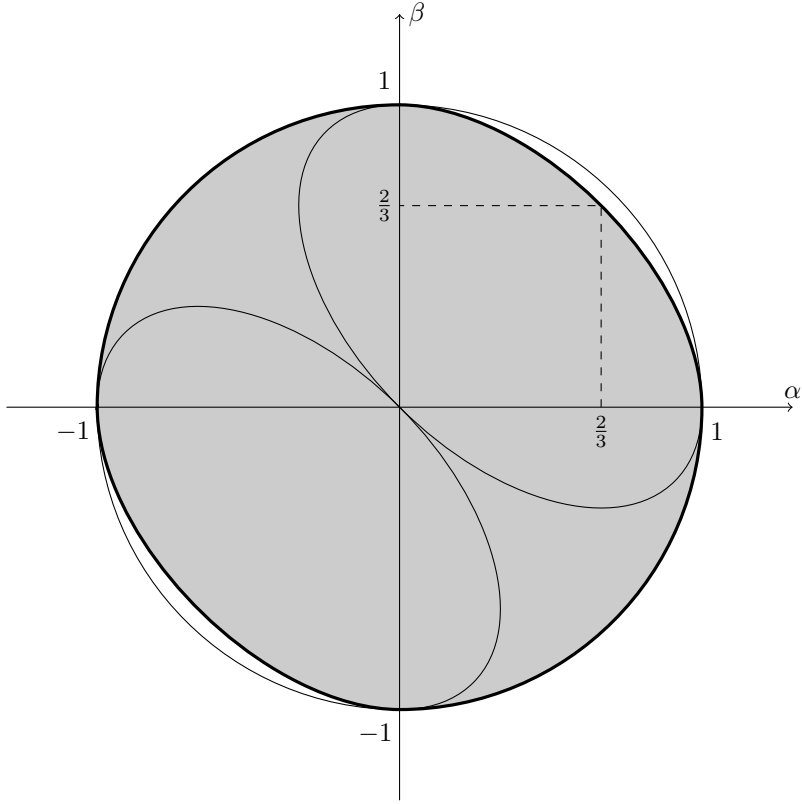
\begin{figure}
\centering
\begin{tikzpicture}[scale=4]
\draw[rotate around={-45:(0,0)}, very thick, fill=black!20] ({1/sqrt(2)},{1/sqrt(2)}) arc (30:150:{sqrt(2/3)} and {sqrt(2)/3}) -- ({-1/sqrt(2)},{1/sqrt(2)}) arc (135:225:1) -- ({-1/sqrt(2)},{-1/sqrt(2)}) arc (210:330:{sqrt(2/3)} and {sqrt(2)/3}) -- ({1/sqrt(2)},{-1/sqrt(2)}) arc (-45:45:1) -- cycle;
\draw (0,0) circle (1);
\draw[->] (-1.3,0)--(1.3,0) node[above]{$\alpha$};
\draw[->] (0,-1.3)--(0,1.3) node[right]{$\beta$};
\node at (1.05,-0.08) {$1$};
\node at (-0.05,1.08) {$1$};
\node at (-1.08,-0.08) {$-1$};
\node at (-0.08,-1.08) {$-1$};
\node at (-0.05, 2/3) {$\frac{2}{3}$};
\node at (2/3,-0.08) {$\frac{2}{3}$};
\draw[rotate around={-45:(0,0)}] (0,{sqrt(2)/3}) ellipse ({sqrt(2/3)} and {sqrt(2)/3});
\draw[rotate around={135:(0,0)}] (0,{sqrt(2)/3}) ellipse ({sqrt(2/3)} and {sqrt(2)/3});
\draw[dashed] (2/3,0) -- (2/3,2/3) -- (0,2/3);
\end{tikzpicture}
\caption{The closure of the set of attainable $(\alpha, \beta)$}
\label{fig:ab}
\end{figure}

\section{Proofs}

\begin{proof}[Proof of Theorem \ref{thmA} $\Rightarrow$ Corollary \ref{corA}]
We apply Theorem \ref{thmA} for the matrix $A = I_n + A_G \in [0,1]^{n \times n}$ with eigenvalues $\lambda_1 + 1 \ge \dotsc \ge \lambda_n + 1$.
As we have seen earlier, $s_1 \ge s_2 \ge 0$, because $1 = \frac{1}{n} \Trace(A) \le s_1 + (n-1) s_2$ and $s_1 \le 1$.
If $G$ is the three-vertex path graph, then $(s_2, s_3) = (\frac{1}{3}, \frac{1-\sqrt{2}}{3})$, and the quadratic inequality for $s_2, s_3$ holds.
Now we prove that if $n \ge 3$ and $G$ is not the three-vertex path graph, then $\lambda_3 \ge -1$ (and thus $s_3 \ge 0$).
The Cauchy interlacing theorem implies that if $H$ is an induced subgraph of $G$ with at least $3$ vertices, then $\lambda_3(H) \le \lambda_3(G)$.
But the only $3$ vertex graph $H$ with $\lambda_3(H) < -1$ is the path graph.
So if $\lambda_3 < -1$, then every three vertex induced subgraph of $G$ has exactly two edges.
This cannot happen for $n \ge 5$, and can only happen for $n = 4$ if $G \cong C_4$, in which case $\lambda_3 = 0$.
So applying Theorem \ref{thmA} for $A = I + A_G$ we get the quadratic bounds.
We also get $\mcalR_{1,2}, \mcalR_{2,3} \subseteq \R_{\ge 0}^2$, because in the definition of $\mcalR_{2,3}$ we require $n_i \to \infty$ as $i \to \infty$.
So $\mcalR_{1,2}$ is a compact subset of $\{(s_1, s_2) \in \R^2; \, 0 \le s_2 \le s_1 \textrm{ and } s_1^2 + s_1 s_2 + 2 s_2^2 \le s_1+s_2 \}$, and $\mcalR_{2,3}$ is a compact subset of $\{(s_2, s_3) \in \R^2; \, 0 \le s_3 \le s_2 \textrm{ and } 2 (s_2^2 + s_2 s_3 + s_3^2) \le s_2+s_3 \}$.
We may add an appropriate number of isolated vertices to the graphs (which only add zero eigenvalues), so it follows that if $c \in [0,1]$ and $(s_1, s_2) \in \mcalR_{1,2}$, then $(c s_1, c s_2) \in \mcalR_{1,2}$ too, and similarly, if $c \in [0,1]$ and $(s_2, s_3) \in \mcalR_{2,3}$, then $(c s_2, c s_3) \in \mcalR_{2,3}$ too.
So it is enough to prove that $(s_1, s_2) \in \mcalR_{1,2}$ for $0 < s_2 < s_1$ with $s_1^2 + s_1 s_2 + 2 s_2^2 = s_1+s_2$, and that $(s_2, s_3) \in \mcalR_{2,3}$ for $0 < s_3 < s_2$ with $2 (s_2^2 + s_2 s_3 + s_3^2) = s_2+s_3$.

For $n', n'' \ge 1$ we define a graph $G_{n', n''}$ with vertex set $[2n'+n'']$.
For $i,j \in [2n'+n'']$ with $i \neq j$, $\{i,j\}$ is an edge in $G_{n',n''}$ if and only if $\max(i,j) \le n'+n''$ or $\min(i,j) \ge n'+1$.
(So $G_{n',n''}$ can be obtained by taking the graph consisting of two disjoint copies of $K_{n'}$, and taking the join of this graph with $K_{n''}$.)
Let $A_{n',n''}$ be the adjacency matrix of $G_{n',n''}$.
It is easy to check that $I_{2n'+n''} + A_{n',n''}$ has only three nonzero eigenvalues, which are the eigenvalues of $\begin{pmatrix} n' & n'' & 0 \\ n' & n'' & n' \\ 0 & n'' & n' \end{pmatrix}$, or equivalently, of $M = \begin{pmatrix} n' & \sqrt{n' n''} & 0 \\ \sqrt{n' n''} & n'' & \sqrt{n' n''} \\ 0 & \sqrt{n' n''} & n' \end{pmatrix}$.
Here $\chi_M(t) = (t - n') (t^2 - (n'+n'') t - n' n'')$, so the three nonzero eigenvalues are
\[
\frac{n'+n'' + \sqrt{(n'+n'')^2 + 4 n' n''}}{2} > n' > \frac{n'+n'' - \sqrt{(n'+n'')^2 + 4 n' n''}}{2}.
\]
First let $G = G_{n', n''}$.
Then $s_1 = \frac{n'+n'' + \sqrt{(n'+n'')^2 + 4 n' n''}}{2 (2n'+n'')}$, $s_2 = \frac{n'}{2n'+n''}$, and $s_1^2 + s_1 s_2 + 2 s_2^2 = s_1+s_2$.
So here $(s_1, s_2)$ is a point of the ellipse arc between $(1,0)$ and $(\frac{1}{2}, \frac{1}{2})$ (see figure \ref{fig:s1s2}), and it only depends on the ratio $\frac{n'}{n''} \in \R_{>0}$.
Moreover $(s_1, s_2)$ converges to $(1,0)$ as $\frac{n'}{n''} \to 0$, and it converges to $(\frac{1}{2}, \frac{1}{2})$ as $\frac{n'}{n''} \to \infty$, and we get every point of the ellipse arc in between.
So $\mcalR_{1,2} = \mcalS_{1,2}$ is indeed the given region.

Now let $G$ be the disjoint union of $G_{n',n''}$ and $K_{n'''}$ (so $G$ has $n = 2n'+n''+n'''$ vertices), where $n''' = \left\lceil \frac{n'+n'' + \sqrt{(n'+n'')^2 + 4 n' n''}}{2} \right\rceil$.
The eigenvalues of $K_{n'''}$ are $n'''-1$ and $-1$ (with multiplicity $n'''-1$).
So the largest $3$ eigenvalues of $I_n + A_G$ are $n''' \ge \frac{n'+n'' + \sqrt{(n'+n'')^2 + 4 n' n''}}{2} > n'$.
So here $s_2 = \frac{n'+n'' + \sqrt{(n'+n'')^2 + 4 n' n''}}{2 n}$ and $s_3 = \frac{n'}{n}$.
If we fix a $q > 0$ and take $n', n'' \to \infty$ and $n'/n'' \to q$, then $\frac{n'''}{n''} \to r = \frac{q + 1 + \sqrt{(q+1)^2 + 4 q}}{2}$, $s_2 \to t_2 := \frac{r}{2q + 1 + r}$ and $s_3 \to t_3 := \frac{q}{2 q + 1 + r}$.
Here $(t_2,t_3)$ is on the ellipse $2 (t_2^2 + t_2 t_3 + t_3^2) = t_2+t_3$, and $(t_2, t_3) \to (\frac{1}{2}, 0)$ as $q \to 0$, and $(t_2, t_3) \to (\frac{1}{3}, \frac{1}{3})$ as $q \to \infty$ (see figure \ref{fig:s2s3}).
So we get that the whole ellipse arc is in $\mcalR_{2,3}$, hence $\mcalR_{2,3} = \mcalS_{2,3} \cap \R_{\ge 0}^2$ is indeed the given region.
The only $G$ for which $s_3<0$ is the three-vertex path graph, so we also get $\mcalS_{2,3} = \mcalR_{2,3} \cup \left\{\left(\frac{1}{3}, \frac{1-\sqrt{2}}{3}\right)\right\}$.
\end{proof}

\begin{proof}[Proof of Propositions \ref{propA12}, \ref{propA23} $\Rightarrow$ Theorem \ref{thmA}]
The matrix $A$ is not negative definite (e.g., because $A_{1,1} \ge 0$), so $s_1 \ge 0$.
There are eigenvectors $u_1, \dotsc, u_n \in \R^n$ such that $|u_i| = 1$, $u_i^T u_j = 0$ for $i \neq j$, and $A u_i = \lambda_i u_i$.
Let $p_i = n^{-1}$, $z_i = \sqrt{n} ((u_1)_i, (u_2)_i) \in \R^2$ (we think of $z_i$ as a column vector).
Then $\sum_i p_i = 1$, $\sum_i p_i z_i z_i^T = (u_i^T u_j)_{i,j \in [2]} = I_2$ and
\[
\Delta = \sum_{i,j} A_{i,j} p_i p_j z_i z_j^T = n^{-1} (u_a^T A u_b)_{a,b \in [2]} = \diag(s_1, s_2),
\]
so $s_1 \le 1$, $s_2 \le \frac{1}{2}$ and the quadratic inequality for $s_1, s_2$ (assuming $s_2 \ge 0$) follows from Proposition \ref{propA12}.
If $n \ge 3$ and $s_3 \ge 0$, then we can take $p_i = n^{-1}$, $z_i = \sqrt{n} ((u_1)_i, (u_2)_i, (u_3)_i) \in \R^3$.
Then $\Delta = \sum_{i,j} A_{i,j} p_i p_j z_i z_j^T = \diag(s_1, s_2, s_3)$, so we get the quadratic inequality for $s_2, s_3$ using Proposition \ref{propA23}.
\end{proof}

\begin{proof}[Proof of Proposition \ref{propB} $\Rightarrow$ Theorem \ref{thmB}]
We use the same method as in the proof of Theorem \ref{thmA}.
There are eigenvectors $u,v \in \R^n$ such that $|u| = |v| = 1$, $u^T v = 0$, $B u = n \alpha u$ and $B v = n \beta v$.
Let $p_i = n^{-1}$, $z_i = \sqrt{n} (u_i, v_i) \in \R^2$ (we think of $z_i$ as a column vector).
Then $\sum_i p_i = 1$, $\sum_i p_i z_i z_i^T = I_2$ and
\[
\Delta = \sum_{i,j} B_{i,j} p_i p_j z_i z_j^T = n^{-1} \begin{pmatrix} u^T B u & u^T B v \\
v^T B u & v^T B v \end{pmatrix} = \begin{pmatrix} \alpha & 0 \\
0 & \beta \end{pmatrix},
\]
so the theorem follows from Proposition \ref{propB}.
\end{proof}

\begin{proof}[Proof of Propositions \ref{propA12}, \ref{propA23}, \ref{propB} $\Rightarrow$ Corollary \ref{corM}]
There are orthogonal, norm $1$ eigenvectors $v_1, v_2 \in \R^n$ of $M$ for the eigenvalues $s_1$, $s_2$.
So $v_1^T v_2 = 0$, $|v_1| = |v_2| = 1$, $M v_1 = s_1 v_1$, $M v_2 = s_2 v_2$.
For $i \in [n]$ let $z_i = (\frac{(v_1)_i}{\sqrt{p_i}}, \frac{(v_2)_i}{\sqrt{p_i}}) \in \R^2$ (column vector) if $p_i > 0$, otherwise let $z_i = 0$.
Then
\[
\sum_i p_i z_i z_i^T = \sum_{i: \, p_i > 0} \begin{pmatrix} (v_1)_i^2 & (v_1)_i (v_2)_i \\ (v_1)_i (v_2)_i & (v_2)_i^2 \end{pmatrix} \preceq \sum_i \begin{pmatrix} (v_1)_i^2 & (v_1)_i (v_2)_i \\ (v_1)_i (v_2)_i & (v_2)_i^2 \end{pmatrix} = I_2,
\]
and
\[
\Delta = \sum_{i,j} A_{i,j} p_i p_j z_i z_j^T = (\sum_{i,j} M_{i,j} (v_a)_i (v_b)_j)_{a,b \in [2]} = (v_a^T M v_b)_{a,b \in [2]} = \diag(s_1, s_2),
\]
so the $(t-s_1) (t-s_2) \mid \chi_M(t)$ case follows from Proposition \ref{propA12}.
In the $(t-s_1) (t-s_2) (t-s_3) \mid \chi_M(t)$ case the same proof works with $z_i = (\frac{(v_1)_i}{\sqrt{p_i}}, \frac{(v_2)_i}{\sqrt{p_i}}, \frac{(v_3)_i}{\sqrt{p_i}}) \in \R^3$ for $p_i > 0$, $z_i = 0$ otherwise, using Proposition \ref{propA23}.
Finally, if $A \in [-1,1]^{n \times n}$ and $(t-\alpha) (t-\beta) \mid \chi_M(t)$, then again the same proof works, with $M v_1 = \alpha v_1$, $M v_2 = \beta v_2$ and $z_i = (\frac{(v_1)_i}{\sqrt{p_i}}, \frac{(v_2)_i}{\sqrt{p_i}}) \in \R^2$ for $p_i > 0$, $z_i = 0$ otherwise, using Proposition \ref{propB}.
\end{proof}

\begin{proof}[Proof of Proposition \ref{propA23} $\Rightarrow$ Proposition \ref{propA12}]
Let $P = \sum_{i=1}^n p_i z_i z_i^T \in \R^{2 \times 2}$ with $P \preceq I_2$.
Let $\Delta = \sum_{i,j \in [n]} A_{i,j} p_i p_j z_i z_j^T \in \R^{2 \times 2}$ have eigenvalues $s_1 \ge s_2$.
The claim is trivial for $s_1 \le 0$, so assume that $s_1 > 0$.
Let $\delta \in (0,1)$, $q_i = \delta p_i$ and $x_i = (\delta^{-\frac{1}{2}} (z_i)_1, \delta^{-\frac{1}{2}} (z_i)_2, 0) \in \R^3$ for $i \in [n]$, and let $x_{n+1} = (0,0, (1-\delta)^{-\frac{1}{2}}) \in \R^3$ and $q_{n+1} = 1-\delta$.
(Here we think of $x_i \in \R^3$ as a column vector.)
Then $\sum_{i=1}^{n+1} q_i = 1$ and $\sum_{i=1}^{n+1} q_i x_i x_i^T = (\begin{smallmatrix} P & 0 \\ 0 & 1 \end{smallmatrix}) \preceq I_3$ (here we are using block matrix notation).
Let $\widetilde{A} \in [0,1]^{(n+1) \times (n+1)}$, where $\widetilde{A}_{i,j} = A_{i,j}$ for $i,j \in [n]$, $\widetilde{A}_{i,n+1} = \widetilde{A}_{n+1,i} = 0$ for $i \in [n]$, and $\widetilde{A}_{n+1,n+1} = 1$.
Then $\widetilde{\Delta} = \sum_{i,j \in [n+1]} \widetilde{A}_{i,j} q_i q_j x_i x_j^T = (\begin{smallmatrix} \delta \Delta & 0 \\ 0 & 1-\delta \end{smallmatrix})$ has eigenvalues $\delta s_1, \delta s_2, 1-\delta$.
Taking $\delta = \frac{1}{1+s_1}$ we get $\delta s_1 = 1-\delta$, so then $\widetilde{\Delta}$ has eigenvalues $\widetilde{s}_1 = \widetilde{s}_2 \ge \widetilde{s}_3$, where $\widetilde{s}_1 = \widetilde{s}_2 = \frac{s_1}{1+s_1}$ and $\widetilde{s}_3 = \frac{s_2}{1+s_1}$.
Applying Proposition \ref{propA23} for $\widetilde{A}$, $(q_i)_{i \in [n+1]}$, $(x_i)_{i \in [n+1]}$, we get that $\widetilde{s}_2 \le \frac{1}{2}$, and that if $s_2 \ge 0$, then $2 (\widetilde{s}_2^2 + \widetilde{s}_2 \widetilde{s}_3 + \widetilde{s}_3^2) \le \widetilde{s}_2+\widetilde{s}_3$.
So $s_1 \le 1$, and if $s_2 \ge 0$, then $2 (s_1^2 + s_1 s_2 + s_2^2) \le (s_1+s_2) (1+s_1)$ and thus $s_1^2 + s_1 s_2 + 2 s_2^2 \le s_1 + s_2$.
The bound $s_2 \le \frac{1}{2}$ follows from $s_2 \le s_1$ and $s_1^2 + s_1 s_2 + 2 s_2^2 \le s_1 + s_2$.
\end{proof}

\begin{proof}[Proof of Proposition \ref{propB} $\Rightarrow$ Proposition \ref{propA23}]
Let $P = \sum_i p_i z_i z_i^T \in \R^{3 \times 3}$ and $w = \sum_i p_i z_i \in \R^3$.
Let us choose $u, v \in \R^3$ such that $|u| = |v| = 1$ and $u^T v = u^T w = v^T w = 0$, and let $W = \R u + \R v$ (so if $w \neq 0$, then $W = w^{\perp}$).
Let $x_i = (u^T z_i, v^T z_i) \in \R^2$ (so we take in the $(u,v)$ coordinate system the orthogonal projection of $z_i$ onto $W$).
Then $\sum_i p_i x_i = 0$, and
\begin{align*}
\sum_i p_i x_i x_i^T &= \begin{pmatrix} \sum_i p_i u^T z_i z_i^T u  & \sum_i p_i u^T z_i z_i^T v \\
\sum_i p_i u^T z_i z_i^T v & \sum_i p_i v^T z_i z_i^T v \end{pmatrix} = \begin{pmatrix} u^T P u  & u^T P v \\
u^T P v & v^T P v \end{pmatrix} \\
&= U^T P U \preceq U^T U = I_2,
\end{align*}
where $U \in \R^{3 \times 2}$ is the matrix with columns $u$ and $v$.

Let $B \in [-1,1]^{n \times n}$, $B_{i,j} = 2 A_{i,j}-1$.
Let
\begin{align*}
\Delta' &= \sum_{i,j} B_{i,j} p_i p_j x_i x_j^T = 2 \sum_{i,j} A_{i,j} p_i p_j x_i x_j^T - (\sum_i p_i x_i) (\sum_j p_j x_j)^T \\
&= 2 \sum_{i,j} A_{i,j} p_i p_j \begin{pmatrix} u^T z_i z_j^T u &  u^T z_i z_j^T v \\
v^T z_i z_j^T u & v^T z_i z_j^T v \end{pmatrix} = 2 \begin{pmatrix} u^T \Delta u &  u^T \Delta v \\
v^T \Delta u & v^T \Delta v \end{pmatrix} \in \R^{2 \times 2}.
\end{align*}
Thinking of $\Delta$ as a quadratic form on $\R^3$, here we get $2$ times its restriction to the plane $W$.
Let $\alpha \ge \beta$ be the eigenvalues of $\Delta'$.
Using Cauchy's interlacing theorem, we get $s_2 \le \frac{\alpha}{2}$ and $s_3 \le \frac{\beta}{2}$.
Applying Proposition \ref{propB} for $(p_i, x_i)_{i \in [n]}$ and $B$, we get that $\alpha^2 + \beta^2 \le 1$, and if $\alpha \beta \ge 0$, then $\alpha^2 + \alpha \beta + \beta^2 \le |\alpha+\beta|$.
So $s_2 \le \frac{\alpha}{2} \le \frac{1}{2}$.
Now suppose that $s_3 \ge 0$.
Then $0 \le 2 s_3 \le \beta \le \alpha \le 1$, so $\alpha^2 + \alpha \beta + \beta^2 \le \alpha+\beta$.
The ellipse $\{(x,y) \in \R^2; x^2 + x y + y^2 \le x+y \}$ is convex and contains $(0,0)$, $(1,0)$, $(0,1)$, $(\alpha, \beta)$, so it contains the rectangle $[0,\alpha] \times [0,\beta]$ (because $\alpha, \beta \in [0,1]$), hence it also contains $(2 s_2, 2 s_3)$.
Thus we get $2 (s_2^2 + s_2 s_3 + s_3^2) \le s_2+s_3$.
\end{proof}

\begin{proof}[Proof of Lemma \ref{lem:ellipse}]
Let $\eta = \alpha^2 + \alpha \beta + \beta^2 - \alpha - \beta$.
Let $x = \frac{3}{2} (\alpha+\beta - \frac{2}{3})$, $y = \frac{\sqrt{3}}{2} (\alpha - \beta)$.
Then $\eta = \frac{1}{3} (x^2+y^2-1)$ and $c_1 \alpha + c_2 \beta = \frac{c_1+c_2}{3} x + \frac{c_1-c_2}{\sqrt{3}} y + \frac{c_1+c_2}{3}$.
Suppose that $\eta \le 0$, and let $c_1, c_2 \ge 0$.
Then $x^2+y^2 \le 1$, so by the Cauchy-Schwarz inequality,
\[
c_1 \alpha + c_2 \beta \le \sqrt{\left(\frac{c_1+c_2}{3}\right)^2 + \left(\frac{c_1-c_2}{\sqrt{3}}\right)^2} + \frac{c_1+c_2}{3} = K_{c_1, c_2}.
\]
Conversely, suppose that $c_1 \alpha + c_2 \beta \le K_{c_1, c_2}$ for every $c_1, c_2 \ge 0$.
If $\alpha+\beta \le 1$, then $\alpha^2 + \alpha \beta + \beta^2 \le (\alpha+\beta)^2 \le \alpha + \beta$.
Now let $\alpha+\beta > 1$.
Then $c_1 = 2\alpha+\beta-1 > 0$ and $c_2 = \alpha+2\beta-1 > 0$, so
\[
2 (3 \eta + 1) = 3(c_1 \alpha + c_2 \beta) - c_1 - c_2 \le 2 \sqrt{c_1^2 - c_1 c_2 + c_2^2} = 2 \sqrt{3 \eta + 1},
\]
thus $\eta \le 0$.
\end{proof}

\begin{proof}[Proof of Proposition \ref{propB2} $\Rightarrow$ Proposition \ref{propB}]
Let $P = \sum_i p_i z_i z_i^T$.
Using the Cauchy-Schwarz inequality we get
\begin{align*}
\alpha^2+\beta^2 &= \Trace(\Delta^2) = \sum_{i,j} B_{i,j} p_i p_j z_i^T \Delta z_j \le \sum_{i,j} p_i p_j |z_i^T \Delta z_j| \\
&\le \left(\sum_{i,j} p_i p_j (z_i^T \Delta z_j)^2\right)^{\frac{1}{2}} = \left(\sum_{i,j} p_i p_j z_i^T \Delta z_j z_j^T \Delta z_i \right)^{\frac{1}{2}} \\
&= \sqrt{\Trace((P \Delta)^2)} \le \sqrt{\Trace(\Delta^2)} = \sqrt{\alpha^2+\beta^2},
\end{align*}
thus $\alpha^2+\beta^2 \le 1$.
Here we have used that for $0 \preceq P \preceq I_2$ and for a symmetric matrix $\Delta$ we always have $\Trace((P \Delta)^2) \le \Trace(\Delta^2)$.
(Proof: We may assume that $P$ is diagonal, because we can replace $(P, \Delta)$ with $(O P O^T, O \Delta O^T)$, where $O$ is an orthogonal matrix.
So let $P$ be diagonal, and let $Q = \sqrt{P}$.
Then $\Trace((P \Delta)^2) = \Trace((Q \Delta Q)^2) = \sum_{i,j} ((Q \Delta Q)_{i,j})^2 = \sum_{i,j} (Q_{i,i} Q_{j,j} \Delta_{i,j})^2 \le \sum_{i,j} \Delta_{i,j}^2 = \Trace(\Delta^2)$.)

Now let $\alpha \beta \ge 0$.
Taking $-B$ instead of $B$ negates $\alpha$ and $\beta$, so we may assume that $\alpha, \beta \ge 0$.
Let $c_1, c_2 \ge 0$, $C = \diag(c_1, c_2)$.
There is an orthogonal matrix $O \in \R^{2 \times 2}$ such that $O \Delta O^T = \diag(\alpha, \beta)$.
Applying Proposition \ref{propB2} for $(p_i)_{i \in [n]}$, $(O z_i)_{i \in [n]}$, $C$, we get
\begin{align*}
c_1 \alpha + c_2 \beta &= \Trace(C O \Delta O^T) = \sum_{i,j} p_i p_j B_{i,j} z_i^T O^T C O z_j \\
& \le \sum_{i,j} p_i p_j |z_i^T O^T C O z_j| \le K_{c_1, c_2}.
\end{align*}
By Lemma \ref{lem:ellipse} it follows that $\alpha^2 + \alpha \beta + \beta^2 \le \alpha + \beta$.
\end{proof}

\section{Proof of Proposition \ref{propB2}}

By continuity, it is enough to prove the claim for $c_1, c_2 > 0$ with $c_1 \neq c_2$.
So let us fix $c_1, c_2 > 0$, $c_1 \neq c_2$, $C = \diag(c_1, c_2)$.
For $p \in \R^n$, $z = (z_i)_{i \in [n]}$ with $z_i \in \R^2$, let $S(p,z) := \sum_{i,j} p_i p_j |z_i^T C z_j|$.
Let $\mcalM'_n$ be the set of $(p,z)$ with $p \in \R_{\ge 0}^n$, $z = (z_i)_{i \in [n]}$, $z_i \in \R^2$, $\sum_i p_i = 1$, $\sum_i p_i z_i z_i^T \preceq I_2$.
Let $\mcalM_n$ be the set of $(p,z) \in \mcalM'_n$ with $\sum_i p_i z_i z_i^T = I_2$.
Let $\sigma_n := \sup_{(p,z) \in \mcalM'_n} S(p,z)$ and $\sigma = \sup_{n \ge 1} \sigma_n$.
Clearly $\sigma_1 \le \sigma_2 \le \sigma_3 \le \dotsc$
We need to prove $\sigma \le K_{c_1, c_2}$.
We assume indirectly that $\sigma > K_{c_1, c_2}$.

An overview of the proof: First we prove that $\sigma_n = \sigma_{n-1}$ for every big enough $n$.
The idea is that if there are many $\pm z_i$ pointing towards nearly the same direction, then we can use ``$p$-variation'': vary the $p_i$'s so that some $p_i$ becomes $0$, while $S(p,z)$ does not decrease.
Next we prove using a compactness argument that there is in fact a $(p,z) \in \mcalM_n$ with some $n \ge 1$ such that $S(p,z) = \sigma$.
We will take $n$ to be minimal.
Once we have a $(p,z) \in \mcalM_n$ with $S(p,z) = \sigma$, we can look at the partial derivatives with respect to $p$, $z$, and use Lagrange multipliers to get a lot of information about $p$, $z$.
This information, together with the $p$-variation and the minimality of $n$, will allow us to prove $n \le 5$, with further restrictions on $p$, $z$.
Finally, we deal with the remaining special cases directly.

Let us call $(J',J'')$ a special pair for $(p,z)$, if $J'$, $J''$ are disjoint subsets of $[n]$, such that $z_i^T C z_j \ge 0$ if $i,j \in J'$ or $i,j \in J''$, and $z_i^T C z_j \le 0$ if $i \in J'$, $j \in J''$.
The following lemma shows that if we have a special pair satisfying a certain condition, then $n$ can be decreased without decreasing $S(p,z)$.

\begin{lemma}[$p$-variation] \label{lem:pvariation}
Let $(p,z) \in \mcalM'_n$.
Let $(J', J'')$ be a special pair for $(p,z)$.
Suppose that there is a $q \in \R^n \setminus \{0\}$ such that $\sum_i q_i = 0$, $\sum_i q_i z_i z_i^T = 0$, and $q_i = 0$ for every $i \in [n] \setminus (J' \cup J'')$.
Then there is a $t \in \R$ and a $k \in [n]$ such that $p+tq \in \R_{\ge 0}^n$, $(p+tq)_k = 0$, and $S(p+tq,z) \ge S(p,z)$.
So there is a $(p',z') \in \mcalM'_{n-1}$ such that $S(p',z') \ge S(p,z)$.
\end{lemma}
\begin{proof}
Let $J = J' \cup J''$.
Then
\[
S(p+tq,z) = \sum_{i,j} (p_i + t q_i) (p_j + t q_j) |z_i^T C z_j| = S(p,z) + \beta t + \alpha t^2
\]
for $t \in \R$, with $\beta = 2 \sum_i q_i \sum_j p_j |z_i^T C z_j|$ and $\alpha = \sum_{i,j \in J} q_i q_j |z_i^T C z_j| = x^T C x \ge 0$, where $x = \sum_{i \in J'} q_i z_i - \sum_{i \in J''} q_i z_i$.
So this is a convex function of $t$.
Since $q \neq 0$, the set of $t \in \R$ such that $p+tq \in \R_{\ge 0}^n$, is a closed interval $[t_0, t_1]$, with $t_0 \le 0 \le t_1$.
So using the convexity we get that $S(p + t_c q, z) \ge S(p,z)$ for some $c \in \{0,1\}$, and here $(p + t_c q)_k = 0$ for some $k \in [n]$.
Then $(p + t_c q, z) \in \mcalM'_n$, and deleting $(p + t_c q)_k$ and $z_k$, we get $(p',z') \in \mcalM'_{n-1}$ with $S(p',z') = S(p + t_c q, z) \ge S(p,z)$.
\end{proof}

If $n > 16$, then there is a closed quadrant of the plane containing at least $5$ of the points $\sqrt{C} z_i$.
Let $J \subseteq [n]$ be the corresponding set of indices $i$.
Then $|J| \ge 5$, and $z_i^T C z_j \ge 0$ for every $i,j \in J$.
Then $\sum_{i \in J} q_i = 0$ and $\sum_{i \in J} q_i z_i z_i^T = 0$ give us $4$ equations for $(q_i)_{i \in J}$ (note that $z_i z_i^T \in \R^{2 \times 2}$ are symmetric matrices), so there is a nonzero solution.
So using the lemma, we get $\sigma_n = \sigma_{n-1}$ for $n > 16$.
So $\sigma_n = \sigma$ for $n \ge 16$.

Next we show that the supremum $\sigma_n$ is achieved.
To see this, we need a change of variables to get compactness.
Fix $n \ge 1$.
Let $y_i = \sqrt{p_i} z_i \in \R^2$ and $S'(p,y) := \sum_{i,j} \sqrt{p_i p_j} |y_i^T C y_j|$.
Then $\sum_i p_i = 1$, $\sum_i y_i y_i^T \preceq I_2$ and $S(p,z) = S'(p,y)$.
Here $y_i y_i^T \preceq I_2$, so $|y_i| \le 1$.
So by compactness and continuity, there is a $(p, y)$ with $p \in \R_{\ge 0}^n$, $y = (y_i)_{i \in [n]}$, $y_i \in \R^2$, $\sum_i p_i = 1$, $\sum_i y_i y_i^T \preceq I_2$ and maximal $S'(p,y)$.
Let $z_i = p_i^{-\frac{1}{2}} y_i$ if $p_i > 0$, otherwise let $z_i = 0$.
Then $(p,z) \in \mcalM'_n$ and $S(p,z) = S'(p,y) = \sigma_n$.

So there is an $n \in \{1, \dotsc, 16\}$ and a $(p,z) \in \mcalM'_n$ such that $S(p,z) = \sigma$.
We choose $n$ to be minimal.
If $p_i = 0$ for some $i$, then we could delete $p_i$ and $z_i$, contradicting the minimality of $n$.
So $p_i > 0$ for every $i$.
We will show now that $(p,z) \in \mcalM_n$.
Suppose indirectly that $\sum_i p_i z_i z_i^T \neq I_2$.
Then there is a $v \in \R^2 \setminus \{0\}$ such that $\sum_i p_i z_i z_i^T \preceq I_2 - v v^T$.
Let $\epsilon \in (0,1)$, $p' \in \R_{\ge 0}^{n+1}$, $p'_i = (1-\epsilon) p_i$ and $z'_i = (1-\epsilon)^{-\frac{1}{2}} z_i$ for $i \in [n]$, $p'_{n+1} = \epsilon$, $z'_{n+1} = \epsilon^{-\frac{1}{2}} v$.
Then $\sum_{i=1}^{n+1} p'_i = 1$, $\sum_{i=1}^{n+1} p'_i z'_i (z'_i)^T = (\sum_{i=1}^n p_i z_i z_i^T) + v v^T \preceq I_2$ and
\begin{align*}
S(p', z') &= (\sum_{i,j \in [n]} p'_i p'_j |(z'_i)^T C z'_j|) + 2 p'_{n+1} (\sum_{i \in [n]} p'_i |(z'_{n+1})^T C z'_i|) \\
&+ (p'_{n+1})^2 |(z'_{n+1})^T C z'_{n+1}| \\
&= (1-\epsilon) S(p,z) + 2 \sqrt{1-\epsilon} \sqrt{\epsilon} (\sum_{i \in [n]} p_i |v^T C z_i|) + \epsilon v^T C v \\
&= \sigma + 2 \sqrt{\epsilon} (\sum_{i \in [n]} p_i |v^T C z_i|) + O(\epsilon)
\end{align*}
at $\epsilon \to 0$.
Here $(p',z') \in \mcalM'_{n+1}$, so $S(p',z') \le \sigma$, therefore $\sum_{i \in [n]} p_i |v^T C z_i| \le 0$.
So $(C v)^T z_i = v^T C z_i = 0$ for every $i$.
Here $C v \neq 0$, so there are $u \in \R^2$, $a_1, \dotsc, a_n \in \R$ such that $|u| = 1$ and $z_i = a_i u$ for every $i$.
Then $\sum_i p_i = 1$, $(\sum_i p_i a_i^2) u u^T \preceq I_2$ and $\sigma = (\sum_i p_i |a_i|)^2 (u^T C u)$.
Then $\sum_i p_i a_i^2 \le 1$, so using the Cauchy-Schwarz inequality, we get $(\sum_i p_i |a_i|)^2 \le \sum_i p_i a_i^2 \le 1$.
So $K_{c_1, c_2} < \sigma \le u^T C u \le \max(c_1, c_2)$.
However it is easy to check that $K_{c_1, c_2} > \max(c_1, c_2)$ for $c_1, c_2>0$.
This contradiction proves that indeed $(p,z) \in \mcalM_n$.

The next step of the proof is to use the method of Lagrange multipliers (a standard tool in constrained optimization) to get further information on the maximizing $(p,z)$.

\begin{lemma}
Suppose that $\sigma > K_{c_1, c_2}$.
Let $n \ge 1$, $(p,z) \in \mcalM_n$ such that $S(p,z) = \sigma$ and $n$ is minimal.
Let $B \in \{-1,1\}^{n \times n}$ be a symmetric matrix such that $B_{i,j} z_i^T C z_j \ge 0$ for every $i,j$.
Let $Z_i = \sum_{j \in [n]} B_{i,j} p_j z_j \in \R^2$ for $i \in [n]$, and let $\Delta = \sum_{i,j} B_{i,j} p_i p_j z_i z_j^T \in \R^{2 \times 2}$.
Then $n \ge 3$, $\Delta$ is diagonal and positive definite, $\Trace(C \Delta) = \sigma$, and $Z_i = \Delta z_i$, $z_i^T C \Delta z_i = \sigma$ for every $i$.
Moreover $p_i > 0$, $z_i \neq 0$ and $B_{i,i} = 1$ for every $i$, and $z_i \neq \pm z_j$ and $Z_i \neq \pm Z_j$ for $i \neq j$.
\end{lemma}
\begin{proof}
If $p_i = 0$ for some $i$, then we can delete $(p_i, z_i)$, contradicting the minimality of $n$.
So $p_i > 0$ for every $i$.
So $(p,z)$ maximizes $\sum_{i,j} p_i p_j B_{i,j} z_i^T C z_j$ under the constraints $\sum_i p_i = 1$, $\sum_i p_i z_i z_i^T = I_2$ with $p_i \in \R_{>0}$, $z_i \in \R^2$, and the maximum value is $\sigma$.
So our objective function is
\[
f \colon \R^n \times (\R^2)^n \to \R, \quad (p,z) \mapsto \sum_{i,j} p_i p_j B_{i,j} z_i^T C z_j,
\]
and our constraint function is
\[
g \colon \R^n \times (\R^2)^n \to \R \times \Sym_2, \quad (p,z) \mapsto (1 - (\sum_i p_i), I_2 - (\sum_i p_i z_i z_i^T)),
\]
where $\Sym_2$ denotes the vector space of symmetric matrices in $\R^{2 \times 2}$.
Note that here $\R \times \Sym_2$ is a $4$ dimensional vector space.
We know that $(p,z)$ is a local maximum point of $f$ under the constraint $g(p,z) = 0$.
In order for the method of Lagrangian multipliers to be applicable, we need to check the constraint gradient requirement: the differential of $g$ at $(p,z)$ has to have rank $4$.
Suppose indirectly that this rank is not $4$.
Then there is a $(c, G) \in (\R \times \Sym_2) \setminus \{(0,0)\}$ such that
\[
\nu(p,z) = c \sum_i p_i + \Trace(G \sum_i p_i z_i z_i^T) = c \sum_i p_i + \sum_i p_i z_i^T G z_i
\]
has zero differential at $(p,z)$.
Then $(\partial_{(z_i)_1} \nu)(p,z) = (\partial_{(z_i)_2} \nu)(p,z) = 0$, so we get (using $p_i \neq 0$) that $G z_i = 0$ for every $i$.
So $G = G (\sum_i p_i z_i z_i^T) = 0$, hence $c \neq 0$.
Then $0 = (\partial_{p_1} \nu)(p,z) = c \neq 0$.
This contradiction shows that the constraint gradient requirement is satisfied.

Our Lagrangian function is
\begin{align*}
&\mcalL(p,z,\mu,\lambda) \\
&:= \sum_{i,j} p_i p_j B_{i,j} z_i^T C z_j + \mu (1 - (\sum_i p_i) ) + \sum_{a,b \in [2]} \lambda_{a,b} (I_2 - (\sum_i p_i z_i z_i^T))_{a,b} \\
&= (\sum_{i,j} p_i p_j B_{i,j} z_i^T C z_j) + \mu (1 - (\sum_i p_i)) - (\sum_i p_i z_i^T \lambda z_i) + \Trace(\lambda),
\end{align*}
where $\mu \in \R$, $\lambda = (\lambda_{a,b})_{a,b \in [2]} \in \Sym_2$.
The method of Lagrange multipliers tells us that there is a unique $(\mu, \lambda) \in \R \times \Sym_2$ such that $\frac{\partial \mcalL}{\partial p}$, $\frac{\partial \mcalL}{\partial z}$, $\frac{\partial \mcalL}{\partial \mu}$, $\frac{\partial \mcalL}{\partial \lambda}$ are all zero at $(p,z,\mu,\lambda)$.
The equations $\frac{\partial \mcalL}{\partial \mu} = 0$ and $\frac{\partial \mcalL}{\partial \lambda} = 0$ just give us the constraints $\sum_i p_i = 1$ and $\sum_i p_i z_i z_i^T = I_2$.
From $\frac{\partial \mcalL}{\partial p_k} = 0$ we get $2 z_k^T C Z_k - \mu - z_k^T \lambda z_k = 0$.
From $\frac{\partial \mcalL}{\partial z_k} = 0$ we get $2 p_k C Z_k - 2 p_k \lambda z_k = 0$, so $C Z_k = \lambda z_k$.
So we get $\mu = z_k^T \lambda z_k = z_k^T C Z_k$.
Then
\[
\Trace(\lambda) = \Trace(\lambda \sum_k p_k z_k z_k^T) = \sum_k p_k z_k^T \lambda z_k = \mu
\]
and $\sigma = \sum_i p_i z_i^T C Z_i = \mu$.
Let $D := C^{-1} \lambda \in \R^{2 \times 2}$.
Then $C D = (C D)^T$, $Z_i = D z_i$ and $z_i^T C D z_i = \sigma$ for every $i$, and $\Trace(C D) = \sigma$.
So $D = D (\sum_j p_j z_j z_j^T) = \sum_j p_j Z_j z_j^T = \sum_{i,j} B_{i,j} p_i p_j z_i z_j^T = \Delta$.
So $\Delta = \Delta^T$ and $C \Delta = \Delta C$.
Using $c_1 \neq c_2$ we get that $\Delta$ is diagonal.

Let $\Delta = \diag(\alpha, \beta)$.
As we have seen earlier (in the proof of Proposition \ref{propB2} $\Rightarrow$ Proposition \ref{propB}), we have $\alpha^2 + \beta^2 = \Trace(\Delta^2) \le 1$, so $\alpha, \beta \le 1$.
If $\alpha \le 0$ or $\beta \le 0$, then $K_{c_1,c_2} < \sigma = \Trace(C \Delta) = c_1 \alpha + c_2 \beta \le \max(c_1, c_2) < K_{c_1,c_2}$.
So $\alpha, \beta > 0$, hence $\Delta \succ 0$.
Since $z_i^T C \Delta z_i = \sigma > 0$, we have $z_i \neq 0$ and thus $B_{i,i} = 1$ for every $i$.
If $z_i = \pm z_j$ for some $i \neq j$, then we can replace $(p_i, z_i)$, $(p_j, z_j)$ with $(p_i+p_j, z_i)$ without changing $S(p,z)$, contradicting the minimality of $n$.
So $z_i \neq \pm z_j$ and thus $Z_i \neq \pm Z_j$ for $i \neq j$ (here we have used that $\Delta$ is invertible).
Since $\sum_i p_i z_i z_i^T = I_2$ has rank $2$, we have $n \ge 2$.
If $n=2$, then $Z_2 = B_{1,2} Z_1 = \pm Z_1$, which is impossible.
So $n \ge 3$.
\end{proof}

Let $\Delta = \diag(\alpha, \beta)$, $M = (B_{i,j} \sqrt{p_i p_j})_{i,j \in [n]} \in \R^{n \times n}$, $u_1 = (\sqrt{p_i} (z_i)_1)_{i \in [n]} \in \R^n$, $u_2 = (\sqrt{p_i} (z_i)_2)_{i \in [n]} \in \R^n$.
Then $M^T = M$, and
\[
(M u_1)_i = \sqrt{p_i} \sum_j B_{i,j} p_j (z_j)_1 = \sqrt{p_i} (Z_i)_1 = \sqrt{p_i} (\Delta z_i)_1 = \alpha (u_1)_i
\]
for every $i$, thus $M u_1 = \alpha u_1$.
Similarly, $M u_2 = \beta u_2$.
Here $\sum_i p_i z_i z_i^T = I_2$ means that $|u_1| = |u_2| = 1$ and $u_1^T u_2 = 0$.
So $\alpha$ and $\beta$ are both eigenvalues of $M$ (and if $\alpha = \beta$ then with multiplicity $\ge 2$).
In other words, $(t-\alpha) (t-\beta)$ divides the characteristic polynomial $\chi_M(t) = \det(t I_n - M)$.

Now we know that $\sum_i q_i z_i z_i^T = 0$ implies $\sum_i q_i = 0$, because $\sigma > K_{c_1, c_2}>0$ and $\Trace(C \Delta \sum_i q_i z_i z_i^T) = \sum_i q_i z_i^T C \Delta z_i = \sigma \sum_i q_i$.
Using this observation, Lemma \ref{lem:pvariation}, and the minimality of $n$, we get that $|J'| + |J''| \le 3$ for every special pair $(J', J'')$ for $(p,z)$.

We can negate any $z_i$, and we can also permute the indices of $z_1, \dotsc, z_n$, without changing $\sum_i p_i z_i z_i^T$ and $S(p,z)$.
So we may assume that
\[
\sqrt{C} z_i = r_i (\cos(\pi \theta_i), \sin(\pi \theta_i))
\]
with $r_i > 0$, $0 \le \theta_1 \le \dotsc \le \theta_n < 1$.
In fact here $\theta_1 < \dotsc < \theta_n$, because if $\theta_i = \theta_j$ for some $i \neq j$, then $\R z_i = \R z_j$, and this is impossible, because $z_i^T C \Delta z_i = z_j^T C \Delta z_j = \sigma > 0$ and $z_i \neq \pm z_j$.
Note that $z_i^T C z_j = r_i r_j \cos(\pi (\theta_i - \theta_j))$, so if $|\theta_i - \theta_j| < \frac{1}{2}$, then $B_{i,j} = 1$, and if $|\theta_i - \theta_j| > \frac{1}{2}$, then $B_{i,j} = -1$.

Suppose that $n = 3$.
If $B_{1,3} = 1$, then $\theta_3 - \theta_1 \le \frac{1}{2}$, so $\theta_2-\theta_1, \theta_3 - \theta_2 < \frac{1}{2}$, so $B_{1,2} = B_{2,3} = 1$ and thus $Z_1 = Z_2$, which is impossible.
So $B_{1,3} = -1$.
We have $\theta_3 - \theta_1 < 1$, so $\theta_2 - \theta_1 < \frac{1}{2}$ or $\theta_3 - \theta_2 < \frac{1}{2}$.
So $B_{1,2} = 1$ or $B_{2,3} = 1$.
If $B_{1,2} = -1$, then $B_{2,3} = 1$ and $Z_2 = Z_3$.
So $B_{1,2} = 1$.
If $B_{2,3} = -1$, then $Z_1 = Z_2$.
So $B_{2,3} = 1$.
So 
\[
B = \left(\begin{smallmatrix}
1 & 1 & -1 \\
1 & 1 & 1 \\
-1 & 1 & 1
\end{smallmatrix}\right) \textrm{ and }
M = \left(\begin{smallmatrix}
p_1 & \sqrt{p_1 p_2} & -\sqrt{p_1 p_3} \\
\sqrt{p_1 p_2} & p_2 & \sqrt{p_2 p_3} \\
-\sqrt{p_1 p_3} & \sqrt{p_2 p_3} & p_3
\end{smallmatrix}\right).
\]
Then $\chi_M(t) = t^3 - t^2 + c$, where $c = 4 p_1 p_2 p_3 \in [0, \frac{4}{27}]$.
We know that $(t-\alpha) (t-\beta)$ divides $\chi_M(t)$, so $t^3-t^2+c = (t-\alpha) (t-\beta) (t-\gamma)$ for some $\gamma \in \R$.
Then $\alpha+\beta+\gamma = 1$, $\alpha \beta + \alpha \gamma + \beta \gamma = 0$ and $\alpha \beta \gamma = -c$.
Substituting $\gamma = 1-\alpha-\beta$ into the second equation, we get $\alpha \beta + (\alpha+\beta) (1-\alpha-\beta) = 0$.
Expanding this, we get $\alpha^2 + \alpha \beta + \beta^2 = \alpha + \beta$ (the same calculation appears in Remark \ref{rem:thmB}).
Using Lemma \ref{lem:ellipse} we get $K_{c_1, c_2} < \sigma = \Trace(C \Delta) = c_1 \alpha + c_2 \beta \le K_{c_1, c_2}$.
So $n \ge 4$.

For $i \in [n-3]$ we have $\theta_{i+3} - \theta_i > \frac{1}{2}$, because otherwise $(\{i,i+1,i+2,i+3\}, \varnothing)$ would be a special pair.
Furthermore, for $i \in \{n-2,n-1,n\}$, we have $\theta_i - \theta_{i+3-n} < \frac{1}{2}$, because otherwise $(\{i, \dotsc, n\}, \{1, \dotsc, i+3-n\})$ would be a special pair.
So we get $\theta_4 - \theta_1 > \frac{1}{2} > \theta_{n-2} - \theta_1$, thus $\theta_4 > \theta_{n-2}$, so $n \le 5$.

Suppose that $n = 4$.
Then $\theta_4 - \theta_1 > \frac{1}{2}$, and $\theta_2 - \theta_1, \theta_3 - \theta_2, \theta_4 - \theta_3 < \frac{1}{2}$.
So $B_{1,4} = -1$, $B_{1,2} = B_{2,3} = B_{3,4} = 1$.
If $B_{1,3} = B_{2,4} = 1$, then $Z_2 = Z_3$.
If $B_{1,3} = B_{2,4} = -1$, then $Z_4 = -Z_1$.
If $B_{1,3} = 1$ and $B_{2,4} = -1$, then $Z_1 = Z_2$.
If $B_{1,3} = -1$ and $B_{2,4} = 1$, then $Z_3 = Z_4$.
Since $Z_i \neq \pm Z_j$ for $i \neq j$, we get $n = 5$.

Finally, for $n = 5$ we have $\theta_4-\theta_1, \theta_5-\theta_2 > \frac{1}{2}$ and $\theta_3 - \theta_1, \theta_4 - \theta_2, \theta_5 - \theta_3 < \frac{1}{2}$.
So $B_{1,4} = B_{1,5} = B_{2,5} = -1$, $B_{1,3} = B_{2,4} = B_{3,5} = B_{1,2} = B_{2,3} = B_{3,4} = B_{4,5} = 1$.
So
\[
B = \left(\begin{smallmatrix}
1 & 1 & 1 & -1 & -1 \\
1 & 1 & 1 & 1 & -1 \\
1 & 1 & 1 & 1 & 1 \\
-1 & 1 & 1 & 1 & 1 \\
-1 & -1 & 1 & 1 & 1
\end{smallmatrix}\right) \textrm{ and } 
M = \left(\begin{smallmatrix}
p_1 & \sqrt{p_1 p_2} & \sqrt{p_1 p_3} & -\sqrt{p_1 p_4} & -\sqrt{p_1 p_5} \\
\sqrt{p_1 p_2} & p_2 & \sqrt{p_2 p_3} & \sqrt{p_2 p_4} & -\sqrt{p_2 p_5} \\
\sqrt{p_1 p_3} & \sqrt{p_2 p_3} & p_3 & \sqrt{p_3 p_4} & \sqrt{p_3 p_5} \\
-\sqrt{p_1 p_4} & \sqrt{p_2 p_4} & \sqrt{p_3 p_4} & p_4 & \sqrt{p_4 p_5} \\
-\sqrt{p_1 p_5} & -\sqrt{p_2 p_5} & \sqrt{p_3 p_5} & \sqrt{p_4 p_5} & p_5
\end{smallmatrix}\right).
\]
Then $\chi_M(t) = t^5 - t^4 + 4 b t^2 - 16 a$, where
\[
a = p_1 p_2 p_3 p_4 p_5, \quad b = p_2 p_4 p_1 + p_4 p_1 p_3 + p_1 p_3 p_5 + p_3 p_5 p_2 + p_5 p_2 p_4.
\]
For us the exact value of $a$, $b$ is not important, we only use $a > 0$.
We know that $\chi_M(\alpha) = \chi_M(\beta) = 0$, where $\alpha, \beta > 0$.
Then for $\alpha \neq \beta$ we can solve the linear system of equations $\chi_M(\alpha) = \chi_M(\beta) = 0$ in $a, b$.
Explicitly, we can eliminate $b$ by taking $\alpha^{-2} \chi_M(\alpha) - \beta^{-2} \chi_M(\beta)$.
We get
\[
0 = \frac{\alpha^{-2} \chi_M(\alpha) - \beta^{-2} \chi_M(\beta)}{\alpha-\beta} = 16 \frac{\alpha+\beta}{\alpha^2 \beta^2} a + \alpha^2 + \alpha \beta + \beta^2 - \alpha - \beta.
\]
So $0 < a = -\frac{1}{16} \frac{\alpha^2 \beta^2}{\alpha+\beta} (\alpha^2 + \alpha \beta + \beta^2 - \alpha - \beta)$, hence $\alpha^2 + \alpha \beta + \beta^2 < \alpha + \beta$.
If $\alpha = \beta$, then $\chi_M(\alpha) = \chi_M'(\alpha) = 0$, so $0 = 2 \chi_M(\alpha) - \alpha \chi_M'(\alpha) = (2 - 3 \alpha) \alpha^4 - 32 a$, hence $0 < a = \frac{1}{32} (2 - 3 \alpha) \alpha^4$, thus $\alpha < \frac{2}{3}$ and again $\alpha^2 + \alpha \beta + \beta^2 < \alpha + \beta$.
Using Lemma \ref{lem:ellipse}, we get $K_{c_1, c_2} < \sigma = c_1 \alpha + c_2 \beta \le K_{c_1, c_2}$.
This final contradiction proves that indeed $\sigma \le K_{c_1, c_2}$.

\section*{Acknowledgement}

No AI was used during the preparation of this article, except at the end for a final check.

\bibliographystyle{plainurl}
\bibliography{Quadratic123eig}

\end{document}